\documentclass[reqno]{amsart}

\usepackage{amsmath,amssymb,amsthm,amsfonts}
\usepackage{hyperref}
\usepackage{mathrsfs}
\usepackage{appendix}
\usepackage{graphicx}
\usepackage{setspace}

\usepackage{color}

\newtheorem{lemma}{Lemma}[section]
\newtheorem{theorem}{Theorem}[section]

\newtheorem{proposition}{Proposition}[section]
\newtheorem{remark}{Remark}[section]

\numberwithin{equation}{section}

\begin{document}
\title[Stability for 2-D Poiseuille flow]{Stability for the 2-D plane Poiseuille flow in finite channel}
\author[Shijin Ding]{Shijin Ding}
\address[S. Ding]{School of Mathematical Sciences, South China Normal University,
Guangzhou, 510631, China}
\email{dingsj@scnu.edu.cn}
\author[Zhilin Lin]{Zhilin Lin}
\address[Z. Lin]{School of Mathematical Sciences, South China Normal University,
Guangzhou, 510631, China}
\email{zllin@m.scnu.edu.cn}


\begin{abstract}
In this paper, we study the stability for 2-D plane Poiseuille flow $(1-y^2,0)$ in a channel $\mathbb{T}\times (-1,1)$ with Navier-slip boundary condition. We prove that if the initial perturbation for velocity field $u_0$ satisfies that $\|u_0\|_{H^{\frac{7}{2}+}} \leq \epsilon_1 \nu^{2/3}$ for some suitable small $0<\epsilon_1 \ll 1$ independent of viscosity coefficient $\nu$, then the solution to the Navier-Stokes equations is global in time and does not transit from the plane Poiseuille flow. This result improves the result of \cite{DL1} from $3/4$ to $2/3$.
\end{abstract}

\maketitle

\vspace{-5mm}

\section{Introduction}
In this paper, we consider the 2-D incompressible Navier-Stokes (N-S) equations in a channel $\Omega:=\mathbb{T}\times (-1,1)$:
\begin{equation}\label{1.1}
\left \{
\begin{array}{lll}
\partial_t v-\nu \Delta v+(v\cdot \nabla)v +\nabla q=0,\\
\nabla \cdot v=0,\\
v|_{t=0}=v_0(x,y),
\end{array}
\right.
\end{equation}
where $\nu=Re^{-1} >0$ is the viscosity, $v(t;x,y)\in \mathbb{R}^2$ is the velocity fields and $q\in \mathbb{R}$ is the pressure.

The plane Poiseuille flow $v_s=(1-y^2,0)$ is a steady solution to the Navier-Stokes equations for  constant pressure gradient $\nabla Q\equiv (q_0,0)$. To study the stability of plane Poiseuille flow, let us introduce the perturbation $u=v-v_s, p=q-Q$, then the perturbed equations around Poiseuille flow read as
\begin{equation}\label{1.2}
\left \{
\begin{array}{lll}
\partial_t u-\nu \Delta u+(u\cdot \nabla)u +(1-y^2)\partial_x u+\left(
\begin{array}{c} -2y u_2 \\ 0 \end{array}\right)+\nabla p =0,\\
\nabla \cdot u=0,\\
u|_{t=0}=u_0(x,y),
\end{array}
\right.
\end{equation}
$p$ is the perturbed pressure.

It is necessary to impose the boundary condition for \eqref{1.2}. In this paper, we consider  the Navier-slip boundary condition
\begin{equation}\label{Naver-slip-bc}
\omega(t;x,\pm 1)=0.
\end{equation}
where $\omega=\partial_y u_1-\partial_x u_2$ is the vorticity of the velocity.

The above problem \eqref{1.2} can be rewritten as 
\begin{equation}\label{vor-ns}
\left \{
\begin{array}{lll}
\partial_t\omega-\nu\Delta \omega+(1-y^2)\partial_x \omega+2\partial_x \Phi=-\nabla\cdot (u\omega),\\
\Delta \Phi=\omega, \ u=(\partial_y,-\partial_x)\Phi,\\
u|_{t=0}=u_0(x,y),
\end{array}
\right.
\end{equation}
where $\Phi$ is the stream function and $\omega=\partial_y u_1-\partial_x u_2$ is the vorticity. In addition, the boundary condition \eqref{Naver-slip-bc} is rewritten as follows 
\begin{equation}\label{Naver-slip-bc-1}
(\Phi,\omega)(t;x,\pm 1)=0.
\end{equation}


Our interest in this paper is to understand the long time dynamics behavior of \eqref{vor-ns} with \eqref{Naver-slip-bc-1} for $0<\nu \ll 1$.

The hydrodynamics stability and behavior with large Reynolds number have been one of important topic in fluid dynamics since Reynolds' experiment \cite{Rey}. In the experiment, Reynolds introduced the Reynolds number $Re$ and found that the laminar flow would transit into turbulence with the increasing of the Reynolds number. The experiment suggests that the stability of the laminar flow is very sensitive to Reynolds number. Indeed, by the standard spectrum analysis, most laminar flows (e.g., plane Couette flow \cite{Rom}) are stable for any Reynolds number, however the experiments indicate that the plane laminar flow would transit at high Reynolds regime \cite{Sch,Yaglom}, which is so-called subcritical transition \cite{Chapman}. And the contradiction between the experiment and spectrum analysis is well known as Sommerfeld paradox, see \cite{Li,Sch,Yaglom} and references therein.

To understand the mechanism of transition or stability of laminar flows, there are lots of aspects to resolve the Sommerfeld paradox, including the classical works by Kelvin, Orr, Sommerfeld, C. C. Lin, see \cite{Kelvin,cclin,Orr,Sch,sommerfeld,Yaglom} for instance. 
In particular, Kelvin \cite{Kelvin} pointed out that the basin of attraction of laminar flows may shrink as the Reynolds number tends to infinity, which will lead nonlinear instability. With this view, Trefethen et al. \cite{Tre} proposed the following {\bf  transition threshold problem}, whose mathematical version is formulated in  \cite{BGM4} as follows:\smallskip

\emph{Given a norm $\Vert \cdot\Vert_X$, determine a $\gamma=\gamma(X)$ such that
$$\Vert u_0\Vert_X \leq \nu^{\gamma} \Longrightarrow \ \ stability,$$
$$\Vert u_0\Vert_X \gg \nu^{\gamma} \Longrightarrow\ \ instability.$$}

\noindent  The exponent $\gamma$ is referred to as the transition threshold in the applied literatures. There are lots of works in physical or applied mathematical views devoted to determining $\gamma$ for some important shear flows, such as Couette flow and Poiseuille flow \cite{LHR,Orszag,Reddy,Yaglom}. However, the works about the rigorous analysis for transition threshold for shear flows are still needed and few.

To study the transition threshold problem, it is necessary to understand the linear stabilization effects of shear flows. In this paper, two stabilization effects of plane Poiseuille flow will be focused on: inviscid damping and enhanced dissipation. These two effects play the key roles in the hydrodynamics stability of shear flows. Inviscid damping and enhanced dissipation were firstly observed by Orr \cite{Orr} and Kelvin \cite{Kelvin} respectively, and these two effects for plane Couette flow even monotone flow are well known \cite{BGM1,BGM2,BGM3,BGM4,Bedrossian5,BMV,BVW,CLWZ,CWZ2,CWZ3,Z1,Z2} for different settings and domains. The results about nonlinear inviscid damping for monotone flow in Euler system can be found in \cite{IJia,MZ1}.

The inviscid damping and enhanced dissipation for non-monotone shear flow such as plane Poiseuille flow and Kolmogorov flow are more challenging. The inviscid damping and enhanced dissipation of Kolmogorov flow have been studied in \cite{Beck,Grenier2,IMM,LWZ2,LX,WZZ3} via different methods. For the non-monotone shear flow, the linear inviscid damping and vorticity depletion for Euler system have been studied by Wei, Zhang and Zhao \cite{WZZ2}, see also Ionescu et al. \cite{IIJia}. We refer to \cite{Beck,Grenier1,LWZ1} for more related results.

For the plane Poiseuille flow with high Reynolds number, the enhanced dissipation in $\mathbb{T}\times \mathbb{R}$ was firstly obtained by Coti Zelati et al. \cite{Coti} via the hypocoercivity method, see also \cite{De-Zo}. For plane Poiseuille flow in a channel $\mathbb{T}\times [-1,1]$ with Navier-slip boundary condition, Ding and Lin \cite{DL1} obtained the enhanced dissipation via resolvent estimate. The linear stability for pipe Poiseuille flow has been proved by Chen, Wei and Zhang \cite{CWZ1}, in which the enhanced dissipation was also obtained in fact by the resolvent estimates.

To understand the nonlinear dynamics stability of shear flow, it is necessary to study the {transition threshold problem} mentioned before. There are lots of works devoted to determining the exponent $\gamma$ for different shear flows via the numerical or asymptotic analysis, see \cite{Chapman,LHR, Reddy} and references therein. Recently, there are many progresses about the rigorous mathematical analysis on the transition threshold problem for different shear flows, including the Couette flow and Kolmogorov flow. For 3-D plane Couette flow, either $X$ is Gevrey class or Sobolev space, the transition threshold satisfies $\gamma\leq 1$ for the case without boundaries, see \cite{BGM1,BGM2,BGM3,WZ}. The result $\gamma\leq 1$ also holds in a finite channel with non-slip boundary condition \cite{CWZ2}. The transition threshold for 2-D Couette flow is proved as $\gamma \leq \frac{1}{2}$ in a channel with non-slip boundary condition \cite{CLWZ}, and $\gamma \leq \frac{1}{3}$ in $\mathbb{T}\times\mathbb{R}$ \cite{MZ,WZ1}. Recently, Coti Zelati et al. showed that the transition threshold for Couette flow in 3-D Bounssinesq equations enjoys $\gamma <1$ \cite{Coti-1}. More interesting related results can be found in \cite{Beck,BGM4,BMV,Grenier1,Wei}.

The results for non-monotone flow are few. 
The transition threshold for Kolmogorov flow is proved as $\gamma \leq \frac{2}{3}+$ for 2-D domain \cite{WZZ3}  and $\gamma \leq \frac{7}{4}$ for 3-D domain \cite{LWZ2}.
For 2-D plane Poiseuille flow in $\mathbb{T}\times \mathbb{R}$, Coti Zelati et al. \cite{Coti} proved that transition threshold is $\gamma \leq \frac{3}{4}+$ via hypocoercivity method, which has been improved to $\gamma \leq \frac{2}{3}+$ by Del Zotto \cite{De-Zo}. For 2-D plane Poiseuille flow in a channel with Navier-slip boundary condition, Ding and Lin \cite{DL1} proved that transition threshold is $\gamma \leq \frac{3}{4}$ by resolvent estimate method. Recently, for 3-D plane Poiseuille flow in a channel $\mathbb{T}\times [-1,1]\times \mathbb{R}$ with non-slip boundary condition, Chen, Ding, Lin and Zhang \cite{CDLZ}proved that transition threshold is $\gamma \leq \frac{7}{4}$, which is accordant with the numerical result by Lundbladh et al. \cite{LHR}.

To state our main results, let us define
$$P_0f(t,y)=\overline{f}(t,y)=\int_{\mathbb{T}}f(t,x,y)\mathrm{d}x, \ f_{\not=}=f-\overline{f}=\sum_{k\in\mathbb{Z}\setminus \{0\}}f_k(t,y)e^{ikx},$$
where $f_k(t,y)$ is the Fourier transform of $f$ in $x$.

The main result is stated as follows, which improves the transition threshold of Navier-slip boundary value problem studied in \cite{DL1}. Precisely, we improve the result of  \cite{DL1} from $\nu^{3/4}$ to $\nu^{2/3}$, which is stated as the following theorem. 
\begin{theorem}\label{nonlinear-slip}
Suppose that $\mathrm{div} u_0=0, u_0\in H^{\frac{7}{2}+}$ with $\|u_0\|_{H^{\frac{7}{2}+}} \leq \epsilon_1 \nu^{2/3}$ for some suitable small $0<\epsilon_1 \ll 1$ independent of $\nu$, then \eqref{vor-ns} with \eqref{Naver-slip-bc} admits a global in time solution $\omega$ with 
\begin{equation}\label{nonlinear-estimate}
\sum_{k\in \mathbb{Z}}\mathcal{E}_k \leq C\nu^{2/3},
\end{equation}
where the energy functional $\mathcal{E}_k$ is defined as 
\begin{equation}\label{energy-fun}
\mathcal{E}_k:=\left \{
\begin{array}{lll}
\|e^{c\nu^{\frac{1}{2}} t} \omega_k\|_{L^\infty L^2}+\nu^{\frac{1}{2}}|k|\|e^{c\nu^{\frac{1}{2}} t} \omega_k\|_{L^2L^2}\\
\qquad +(\nu |k|)^{1/4}\|e^{c\nu^{\frac{1}{2}} t} \omega_k\|_{L^2 L^2}+|k|^{1/2}\|e^{c\nu^{\frac{1}{2}} t} u_k\|_{L^2 L^2},& k\not=0,\\
\|\overline{\omega}\|_{L^\infty L^2},&k=0.
\end{array}
\right.
\end{equation}

\end{theorem} 

\begin{remark}
Indeed, by the stability estimate \eqref{nonlinear-estimate}, we can deduce that 
$$\sum_{k\in \mathbb{Z}}\|e^{c\nu^{\frac{1}{2}} t} u_k\|_{L^\infty L^\infty}\leq C\nu^{2/3} ,$$
which implies that the 2-D plane Poiseuille flow is absolutely stable.
\end{remark}

\begin{remark}
Even for the Navier-slip boundary value problem, the transition threshold $\gamma \leq 2/3$ should not be optimal. It is very interesting and challenging to study the optimal transition threshold of the Poiseuille flow. 
\end{remark}

Let us give some notations to end this section. Throughout this paper, we always assume that $|k|\geq 1$ and denote by $C$ a positive constant independent of $\nu,k$. We also use $a\sim b$ for $C^{-1}b\leq a\leq Cb$ and $a\lesssim b$ for $a\leq Cb$ for some constants $C>0$ independent of $\nu,k$.

\section{Sketch and key ideas of the proof}\label{outline}

In this section, let us give the sketch and some key ideas in our analysis. 

In our analysis, compared with \cite{DL1}, the most important point is the improved space-time estimates. Indeed, in the view of Fourier transform, the toy model for our problem (linear and nonlinear sense) reads
\begin{equation}\label{toy-model}
\left \{
\begin{array}{lll}
\partial_t \omega+\mathcal{L}_k \omega=-ikf_1-\partial_y f_2-f_3-f_4,\\
(\partial_y^2-k^2)\varphi=\omega, \ u=(\partial_y,-ik)\varphi,\\
 (\varphi,\omega)|_{y=\pm 1}=(0,0),\ \omega|_{t=0}=\omega_0(k,y),
\end{array}
\right.
\end{equation}
where $\omega_0(k,\pm 1)=0$ and $k\not=0$, the linearized operator is defined as 
$$\mathcal{L}_k:=-\nu (\partial_y^2-k^2)+ik(1-y^2)+2ik(\partial_y^2-|k|^2)^{-1}.$$

Let us recall that the space-time estimates in \cite{DL1} (in the Fourier version) read as 
\begin{equation}\nonumber
\begin{aligned}
\|e^{c'(\nu |k|)^{1/2}t}\omega\|_{L^\infty L^2}^2+(\nu |k|)^{1/2} &\|e^{c'(\nu |k|)^{1/2}t}\omega\|_{L^2 L^2}^2 +\nu \|e^{c'(\nu |k|)^{1/2}t}(\partial_y,|k|)\omega\|_{L^2 L^2}^2\\
\lesssim& \|\omega_0\|_{L^2}^2+\nu^{-1}\|e^{c'(\nu |k|)^{1/2}t}(f_1,f_2)\|_{L^2L^2}^2,
\end{aligned}
\end{equation}
where the term $(\nu |k|)^{1/2} \|e^{c'(\nu |k|)^{1/2}t}\omega\|_{L^2 L^2}^2$ is due to enhanced dissipation, which is resulted from the resolvent estimate obtained in \cite{DL1}. In a word, the stabilization mechanism used in \cite{DL1} is only dissipation, including the enhanced dissipation and heat diffusion. 

To improve the result of \cite{DL1} to $\gamma\leq 2/3$, the inviscid damping effect is added in the space-time estimates in this paper. More precisely, the following space-time estimates are used in this paper
\begin{equation}\nonumber
\begin{aligned}
& \|e^{c \nu^{\frac{1}{2}}t} \omega\|_{L^\infty L^2}^2+\nu\|e^{c \nu^{\frac{1}{2}}t} (\partial_y,|k|)\omega\|_{L^2 L^2}^2\\
&+ (\nu |k|)^{\frac{1}{2}}\|e^{c \nu^{\frac{1}{2}}t} \omega\|_{L^2 L^2}^2+ |k|\|e^{c \nu^{\frac{1}{2}}t}u\|_{L^2 L^2}^2\\
\lesssim& \|\Delta_k \omega_0\|_{L^2}^2+\nu^{-1} \|e^{c \nu^{\frac{1}{2}}t}(f_1,f_2)\|_{L^2 L^2}^2+|k|^{-1}\|e^{c \nu^{\frac{1}{2}}t} (\partial_y,|k|)f_3\|_{L^2 L^2}^2\\
&+\min\{(\nu |k|)^{-\frac{1}{2}},(\nu |k|^2)^{-1}\}\|e^{c \nu^{\frac{1}{2}}t} f_4\|_{L^2 L^2}^2.
\end{aligned}
\end{equation}
Compared with the space-time estimates in \cite{DL1}, in the above estimates, the part $|k|\|e^{c\nu^{\frac{1}{2}}t} u\|_{L^2 L^2}^2$ is due to the inviscid damping effect of the Poiseuille flow. 

To achieve the desired space-time estimates, following the idea of \cite{CDLZ}, the key point is to establish the resolvent estimates for the linearized operator around the Poiseuille flow, see Section \ref{sec:resolvent} for details. With the resolvent estimates at hands, one can achieve the space-time estimates for the linearized N-S system. 

For the nonlinear problem, the energy estimates for non-zero modes read as 
\begin{align*}
\mathcal{E}_k \lesssim  \|\Delta_k \omega_{0,k}\|_{L^2}+ \nu^{-\frac{3}{8}} |k|^{\frac{3}{8}}  \|e^{c\nu^{\frac{1}{2}} t} (f_1)_k\|_{L^2 L^2}+\nu^{-\frac{1}{2}}\| e^{c\nu^{\frac{1}{2}} t}(f_2)_k\|_{L^2 L^2},
\end{align*}
where $\mathcal{E}_k$ is defined as
 \begin{align*}
\mathcal{E}_k:=&\|e^{c\nu^{\frac{1}{2}} t} \omega_k\|_{L^\infty L^2}+\nu^{\frac{1}{2}}|k|\|e^{c\nu^{\frac{1}{2}} t} \omega_k\|_{L^2L^2}\\
& +(\nu |k|)^{1/4}\|e^{c\nu^{\frac{1}{2}} t} \omega_k\|_{L^2 L^2}+|k|^{1/2}\|e^{c\nu^{\frac{1}{2}} t} u_k\|_{L^2 L^2}.\end{align*}
In the above functional, the terms $\|e^{c\nu^{\frac{1}{2}} t} \omega_k\|_{L^\infty L^2}$ and $\nu^{\frac{1}{2}}|k|\|e^{c\nu^{\frac{1}{2}} t} \omega_k\|_{L^2L^2}$ are corresponding to the basic energy and heat diffusion, the term $(\nu |k|)^{1/4}\|e^{c\nu^{\frac{1}{2}} t} \omega_k\|_{L^2 L^2}$ is due to the enhanced dissipation, and $|k|^{1/2}\|e^{c\nu^{\frac{1}{2}} t} u_k\|_{L^2 L^2}$ is due to the inviscid damping effect of Poiseuille flow.

With the nonlinear estimates, the nonlinear stability follows from a continuous argument.

\section{Resolvent estimates for Orr-Sommerfeld equation}\label{sec:resolvent}
In this section, we study the resolvent estimates for the Orr-Sommerfeld equation around the plane Poiseuille flow as follows 
\begin{equation}\label{os-1}
\left \{
\begin{array}{lll}
-\nu (\partial_y^2 -|k|^2)w+ik[(1-y^2-\lambda)w+2\varphi]=F,\\
(\partial_y^2-|k|^2)\varphi=w, \ \varphi(\pm 1)=w(\pm 1)=0,
\end{array}
\right.
\end{equation}
where $k\neq 0$ and $F\in H^1_0$.

The resolvent estimates for \eqref{os-1} are stated as follows.
\begin{proposition}\label{resolvent-slip}
Let $w$ be the solution to \eqref{os-1}, then it holds that 
\begin{equation}\label{resol-slip}
\begin{aligned}
 \nu^{\frac{3}{8}}& |k|^{\frac{5}{8}} (\|w\|_{L^1}+|k|^{\frac{1}{2}}\|u\|_{L^2})\\
&+(\nu |k|)^{\frac{1}{2}}\|w\|_{L^2}+\nu^{\frac{3}{4}}|k|^{\frac{1}{4}}\|(\partial_y,|k|)w\|_{L^2}\lesssim \|F\|_{L^2},\\
\nu^{\frac{3}{4}}&|k|^{\frac{1}{4}} \|w\|_{L^2}+\nu \|(\partial_y,|k|)w\|_{L^2}+(\nu |k|)^{\frac{1}{2}}\|u\|_{L^2} \lesssim \|F\|_{H^{-1}_k},\\
|& \nu/k|^{\frac{1}{2}} \|(\partial_y,|k|)w\|_{L^2}+|\nu/k|^{\frac{1}{4}}\|w\|_{L^2} \\
&+|\nu/k|^{\frac{1}{8}}\|u\|_{L^\infty}+\|u\|_{L^2}\lesssim |k|^{-1} \|(\partial_y,|k|)F\|_{L^2}.
\end{aligned}
\end{equation}

Moreover, there holds that 
\begin{equation}\label{resol-slip-lambda}
\begin{aligned}
\nu^{\frac{1}{6}} & |k|^{\frac{5}{6}} (|\lambda-1|^{\frac{1}{2}}+|\nu/k|^{\frac{1}{4}})^{\frac{1}{3}}\|u\|_{L^2}\\
&+ \nu^{\frac{2}{3}}  |k|^{\frac{1}{3}} (|\lambda-1|^{\frac{1}{2}}+|\nu/k|^{\frac{1}{4}})^{\frac{1}{3}}\|(\partial_y,|k|)w\|_{L^2}\\
 &+ \nu^{\frac{1}{3}}|k|^{\frac{2}{3}} (|\lambda-1|^{\frac{1}{2}}+|\nu/k|^{\frac{1}{4}})^{\frac{2}{3}}\|w\|_{L^2}\lesssim \|F\|_{L^2},\\
 \nu^{\frac{2}{3}} & |k|^{\frac{1}{3}} (|\lambda-1|^{\frac{1}{2}}+|\nu/k|^{\frac{1}{4}})^{\frac{1}{3}}\| w\|_{L^2}\lesssim \|F\|_{H^{-1}_k}.
\end{aligned}
\end{equation}

Here, $\|F\|_{H^{-1}_k}:=\inf\limits_{\{g\in H^1_0:\|g\|_{H^1_k}=1\}}\langle f,g\rangle$ with $\|g\|_{H^1_k}=\|(\partial_y,|k|)g\|_{L^2}.$
\end{proposition}
\begin{proof}
See \cite{CDLZ}.
\end{proof}

\section{Space-time estimates for linearized Navier-Stokes equations around the Poiseuille flow}\label{sec:sp-slip}
The linearized Navier-Stokes equations around Poiseuille flow read as 
\begin{equation}\label{toy-sp-slip}
\left \{
\begin{array}{lll}
\partial_t \omega-\nu(\partial_y^2-k^2)\omega+ik(1-y^2)\omega+2ik\varphi=-ikf_1-\partial_y f_2-f_3-f_4,\\
(\partial_y^2-k^2)\varphi=\omega, \ \ u=(\partial_y,-ik)\varphi,\\
\varphi(t,k,\pm 1)=\omega (t,k,\pm 1)=0,\\
\omega|_{t=0}=\omega_0(k,y),
\end{array}
\right.
\end{equation}
where $f_3|_{y=\pm 1}=0$ and $k\not=0$.

The space-time estimates for \eqref{toy-sp-slip} are stated as follows.

\begin{proposition}\label{sp-slip-0}
Suppose the $\omega_0 \in H^1_0\cap H^2$, then the solution $\omega$ to \eqref{toy-sp-slip} enjoys that the following estimates:
\begin{equation}\label{sp-slip}
\begin{aligned}
& \| \omega\|_{L^\infty L^2}^2+\nu\|(\partial_y,|k|)\omega\|_{L^2 L^2}^2+ (\nu |k|)^{\frac{1}{2}}\|\omega\|_{L^2 L^2}^2+ |k|\|u\|_{L^2 L^2}^2\\
\lesssim& \|\Delta_k \omega_0\|_{L^2}^2+\nu^{-1} \|(f_1,f_2)\|_{L^2 L^2}^2\\
&+|k|^{-1}\|(\partial_y,|k|)f_3\|_{L^2 L^2}^2+\min\{(\nu |k|)^{-\frac{1}{2}},(\nu |k|^2)^{-1}\}\|f_4\|_{L^2 L^2}^2.
\end{aligned}
\end{equation}

Moreover, there exists $0<c\ll 1$ such that 
\begin{equation}\label{sp-slip-weight}
\begin{aligned}
& \|e^{c \nu^{\frac{1}{2}}t} \omega\|_{L^\infty L^2}^2+\nu\|e^{c \nu^{\frac{1}{2}}t} (\partial_y,|k|)\omega\|_{L^2 L^2}^2\\
&+ (\nu |k|)^{\frac{1}{2}}\|e^{c \nu^{\frac{1}{2}}t} \omega\|_{L^2 L^2}^2+ |k|\|e^{c \nu^{\frac{1}{2}}t}u\|_{L^2 L^2}^2\\
\lesssim& \|\Delta_k \omega_0\|_{L^2}^2+\nu^{-1} \|e^{c \nu^{\frac{1}{2}}t}(f_1,f_2)\|_{L^2 L^2}^2+|k|^{-1}\|e^{c \nu^{\frac{1}{2}}t} (\partial_y,|k|)f_3\|_{L^2 L^2}^2\\
&+\min\{(\nu |k|)^{-\frac{1}{2}},(\nu |k|^2)^{-1}\}\|e^{c \nu^{\frac{1}{2}}t} f_4\|_{L^2 L^2}^2.
\end{aligned}
\end{equation}
\end{proposition}

To establish Proposition \ref{sp-slip-0}, let us introduce 
$$\omega=\omega_I+\omega_H$$
with
\begin{equation}\label{toy-sp-slip-I}
\left \{
\begin{array}{lll}
\partial_t \omega_I-\nu(\partial_y^2-k^2)\omega_I+ik(1-y^2)\omega_I+2ik\varphi_I=-ikf_1-\partial_y f_2-f_3-f_4,\\
(\partial_y^2-k^2)\varphi_I=\omega_I, \ \ u_I=(\partial_y,-ik)\varphi_I,\\
\varphi_I(t,k,\pm 1)=\omega_I (t,k,\pm 1)=0,\\
\omega_I|_{t=0}=0,
\end{array}
\right.
\end{equation}
and
\begin{equation}\label{toy-sp-slip-H}
\left \{
\begin{array}{lll}
\partial_t \omega_H-\nu(\partial_y^2-k^2)\omega_H+ik(1-y^2)\omega_H+2ik\varphi_H=0,\\
(\partial_y^2-k^2)\varphi_H=\omega_H, \ \ u_H=(\partial_y,-ik)\varphi_H,\\
\varphi_H (t,k,\pm 1)=\omega_H (t,k,\pm 1)=0,\\
\omega_H|_{t=0}=\omega_0.
\end{array}
\right.
\end{equation}

The space-time estimates for $\omega_I$ are stated as follows.
\begin{lemma}\label{sp-I}
Let $\omega_I$ be the solution to \eqref{toy-sp-slip-I}, then it holds that 
\begin{equation}\nonumber
\begin{aligned}
&\|\omega_I\|_{L^\infty L^2}^2+(\nu |k|)^{\frac{1}{2}}\|\omega_I\|_{L^2 L^2}^2+\nu \|(\partial_y,|k|)\omega_I\|_{L^2 L^2}^2+|k|\|u_I\|_{L^2 L^2}^2\\
\lesssim& \nu^{-1}\| (f_1,f_2)\|_{L^2 L^2}^2+|k|^{-1} \|(\partial_y,|k|)f_3\|_{L^2}^2+\min\{(\nu |k|)^{-\frac{1}{2}},(\nu |k|^2)^{-1}\}\|f_4\|_{L^2 L^2}^2.
\end{aligned}
\end{equation}

In addition, there exists $0<c\ll 1$ such that 
\begin{equation}\label{sp-slip-1-weight}
\begin{aligned}
&\|e^{c\nu^{\frac{1}{2}}t} \omega_I\|_{L^\infty L^2}^2+(\nu |k|)^{\frac{1}{2}}\|e^{c\nu^{\frac{1}{2}}t}  \omega_I\|_{L^2 L^2}^2\\
&\qquad +\nu \|e^{c\nu^{\frac{1}{2}}t}  (\partial_y,|k|)\omega_I\|_{L^2 L^2}^2+|k|\|e^{c\nu^{\frac{1}{2}}t} u_I\|_{L^2 L^2}^2\\
\lesssim& \nu^{-1}\| e^{c\nu^{\frac{1}{2}}t}  (f_1,f_2)\|_{L^2 L^2}^2\\
&+|k|^{-1} \|e^{c\nu^{\frac{1}{2}}t}  (\partial_y,|k|)f_3\|_{L^2}^2+\min\{(\nu |k|)^{-\frac{1}{2}},(\nu |k|^2)^{-1}\}\|e^{c\nu^{\frac{1}{2}}t}  f_4\|_{L^2 L^2}^2.
\end{aligned}
\end{equation}
\end{lemma}

\begin{proof}
Note that the proof is based on the resolvent estimates for $\mathcal{L}_k-\epsilon (\nu |k|)^{\frac{1}{2}}$, we only give the proof for the estimates without $e^{c\nu^{\frac{1}{2}}t}$, and the case with the weight can be obtained via perturbation argument.

By Proposition \ref{resolvent-slip}, we have 
\begin{equation}\nonumber
\begin{aligned}
&(\nu |k|)^{\frac{1}{2}}\|\omega_I\|_{L^2 L^2}^2+\nu \|(\partial_y,|k|)\omega_I\|_{L^2 L^2}^2\\
\lesssim& \nu^{-1}\| (f_1,f_2)\|_{L^2 L^2}^2+|k|^{-1} \|(\partial_y,|k|)f_3\|_{L^2}^2\\
&+\min\{(\nu |k|)^{-\frac{1}{2}},(\nu |k|^2)^{-1}\}\|f_4\|_{L^2 L^2}^2,
\end{aligned}
\end{equation}
where the following energy estimates have been used
\begin{equation}\nonumber
\begin{aligned}
&\nu \|(\partial_y,|k|)\omega_I\|_{L^2 L^2}^2\\
\lesssim& \nu^{-1}\| (f_1,f_2)\|_{L^2 L^2}^2+|k|^{-1} \|(\partial_y,|k|)f_3\|_{L^2}^2+(\nu |k|^2)^{-1}\|f_4\|_{L^2 L^2}^2
\end{aligned}
\end{equation}
and 
\begin{equation}\nonumber
\begin{aligned}
&\nu |k|^2 \| \omega_I\|_{L^2 L^2}^2\\
\lesssim& \nu^{-1}\| (f_1,f_2)\|_{L^2 L^2}^2+|k|^{-1} \|(\partial_y,|k|)f_3\|_{L^2}^2+\|f_4\|_{L^2 L^2}\|\omega\|_{L^2 L^2}\\
\lesssim& \nu^{-1}\| (f_1,f_2)\|_{L^2 L^2}^2+|k|^{-1} \|(\partial_y,|k|)f_3\|_{L^2}^2+(\nu |k|)^{-\frac{1}{2}} \|f_4\|_{L^2 L^2}^2.
\end{aligned}
\end{equation}

It remains to bound $\|u_I\|_{L^2 L^2}$ and $\|\omega_I\|_{L^\infty L^2}$. Let us introduce 
$$\omega_I=\omega_I^{(1)}+\omega_I^{(2)}$$
with
\begin{equation}\label{toy-sp-slip-I-1}
\left \{
\begin{array}{lll}
\partial_t \omega_I^{(1)}-\nu(\partial_y^2-k^2)\omega_I^{(1)}+ik(1-y^2)\omega_I^{(1)}+2ik\varphi_I^{(1)}=-ikf_1-\partial_y f_2-f_3,\\
\partial_t \omega_I^{(2)}-\nu(\partial_y^2-k^2)\omega_I^{(2)}+ik(1-y^2)\omega_I^{(2)}+2ik\varphi_I^{(2)}=f_4,\\
(\partial_y^2-k^2)\varphi_I^{(i)}=\omega_I^{(i)}, \ \ u_I^{(i)}=(\partial_y,-ik)\varphi_I^{(i)},\\
\varphi_I^{(i)}(t,k,\pm 1)=\omega_I^{(i)}(t,k,\pm 1)=0,\ \omega_I^{(i)}|_{t=0}=0,
\end{array}
\right.
\end{equation}
 By Proposition \ref{resolvent-slip}, we have 
\begin{equation}\nonumber
\begin{aligned}
|k| \| u_I^{(1)}\|_{L^2 L^2}^2
\lesssim& \nu^{-1}\| (f_1,f_2)\|_{L^2 L^2}^2+|k|^{-1} \|(\partial_y,|k|)f_3\|_{L^2}^2
\end{aligned}
\end{equation}
and 
\begin{equation}\nonumber
\begin{aligned}
|k| \| u_I^{(2)} \|_{L^2 L^2}^2+(\nu |k|)^{\frac{1}{2}}\|\omega_I^{(2)}\|_{L^2 L^2}^2
\lesssim (\nu |k|)^{-\frac{1}{2}} \|f_4\|_{L^2 L^2}^2.
\end{aligned}
\end{equation}
Meanwhile, one has 
\begin{equation}\nonumber
\begin{aligned}
|k| \| u_I^{(2)} \|_{L^2 L^2}^2\leq |k|^{-1} \|\omega_I^{(2)}\|_{L^2 L^2}^2
\lesssim (\nu |k|^2)^{-1} \|f_4\|_{L^2 L^2}^2.
\end{aligned}
\end{equation}
Thus, it holds that 
\begin{equation}\nonumber
\begin{aligned}
|k| \| u_I^{(2)} \|_{L^2 L^2}^2
\lesssim \min\{(\nu |k|)^{-\frac{1}{2}}, (\nu |k|^2)^{-1}\} \|f_4\|_{L^2 L^2}^2.
\end{aligned}
\end{equation}

Putting the above estimates together, the estimate for $\|u_I\|_{L^2 L^2}$ follows.

Finally, we give the estimate for $\|\omega_I\|_{L^\infty L^2}$. Indeed, we have 
\begin{equation}\nonumber
\begin{aligned}
&\frac{1}{2}\frac{\mathrm{d}}{\mathrm{d}t}\|\omega_I\|_{L^2}^2+\nu \|(\partial_y,|k|)\omega_I\|_{L^2}^2=\mathrm{Re}\langle -ik f_1-\partial_y f_2-f_3-f_4,\omega_I\rangle\\
\leq& \|(f_1,f_2)\|_{L^2}\|(\partial_y,|k|)\omega_I\|_{L^2}+\|(\partial_y,|k|)f_3\|_{L^2}\|(\partial_y,|k|)\varphi_I\|_{L^2}+\|f_4\|_{L^2}\|\omega_I\|_{L^2}\\
\leq& \frac{1}{2} \|(\partial_y,|k|)\omega_I\|_{L^2}^2+\nu^{-1} \|(f_1,f_2)\|_{L^2}^2+|k|^{-1} \|(\partial_y,|k|)f_3\|_{L^2}^2+|k|\|(\partial_y,|k|)\varphi_I\|_{L^2}^2\\
&+(\nu |k|^2)^{-1}\|f_4\|_{L^2}^2
\end{aligned}
\end{equation} 
or 
\begin{equation}\nonumber
\begin{aligned}
&\frac{1}{2}\frac{\mathrm{d}}{\mathrm{d}t}\|\omega_I\|_{L^2}^2+\nu \|(\partial_y,|k|)\omega_I\|_{L^2}^2=\mathrm{Re}\langle -ik f_1-\partial_y f_2-f_3-f_4,\omega_I\rangle\\
\leq& \|(f_1,f_2)\|_{L^2}\|(\partial_y,|k|)\omega_I\|_{L^2}+\|(\partial_y,|k|)f_3\|_{L^2}\|(\partial_y,|k|)\varphi_I\|_{L^2}+\|f_4\|_{L^2}\|\omega_I\|_{L^2}\\
\leq& \frac{1}{2} \|(\partial_y,|k|)\omega_I\|_{L^2}^2+\nu^{-1} \|(f_1,f_2)\|_{L^2}^2+|k|^{-1} \|(\partial_y,|k|)f_3\|_{L^2}^2+|k|\|(\partial_y,|k|)\varphi_I\|_{L^2}^2\\
&+\frac{1}{2}\left((\nu |k|)^{-1/2}\|f_4\|_{L^2}^2+(\nu |k|)^{\frac{1}{2}}\|\omega_I\|_{L^2}^2\right),
\end{aligned}
\end{equation}
which gives that 
\begin{equation}\nonumber
\begin{aligned}
&\frac{1}{2}\frac{\mathrm{d}}{\mathrm{d}t}\|\omega_I\|_{L^2}^2+\nu \|(\partial_y,|k|)\omega_I\|_{L^2}^2\\
\lesssim& \nu^{-1} \|(f_1,f_2)\|_{L^2}^2+|k|^{-1} \|(\partial_y,|k|)f_3\|_{L^2}^2\\
&+\min\{(\nu |k|)^{-1/2},(\nu |k|^2)^{-1}\}\|f_4\|_{L^2}^2+|k|\|u_I\|_{L^2}^2+(\nu |k|)^{\frac{1}{2}}\|\omega_I\|_{L^2}^2.
\end{aligned}
\end{equation}
Therefore, we have 
\begin{equation}\nonumber
\begin{aligned}
&\|\omega_I\|_{L^\infty L^2}^2+\nu \|(\partial_y,|k|)\omega_I\|_{L^2 L^2}^2\\
\lesssim& \nu^{-1} \|(f_1,f_2)\|_{L^2 L^2}^2+|k|^{-1} \|(\partial_y,|k|)f_3\|_{L^2 L^2}^2\\
&+\min\{(\nu |k|)^{-1/2},(\nu |k|^2)^{-1}\}\|f_4\|_{L^2 L^2}^2+|k|\|u_I\|_{L^2 L^2}^2+(\nu |k|)^{\frac{1}{2}}\|\omega_I\|_{L^2 L^2}^2\\
\lesssim&\nu^{-1} \|(f_1,f_2)\|_{L^2 L^2}^2+|k|^{-1} \|(\partial_y,|k|)f_3\|_{L^2 L^2}^2\\
&+\min\{(\nu |k|)^{-1/2},(\nu |k|^2)^{-1}\}\|f_4\|_{L^2 L^2}^2.
\end{aligned}
\end{equation}
The proof is completed.
\end{proof}

Next we give the estimates for $\omega_H$.
\begin{lemma}
It holds that 
\begin{equation}\nonumber
\begin{aligned}
\|&\omega_H\|_{L^\infty L^2}^2+\nu \|(\partial_y,|k|)\omega_H\|_{L^2 L^2}^2\\
&+(\nu |k|)^{\frac{1}{2}}\|\omega_H\|_{L^2 L^2}^2+|k| \| u_H \|_{L^2 L^2}^2
\lesssim \|\Delta_k \omega_0\|_{L^2}^2.
\end{aligned}
\end{equation} 
In addition, there exists $0<c\ll 1$ such that 
\begin{equation}\nonumber
\begin{aligned}
\|e^{c\nu^{\frac{1}{2}}t} & \omega_H\|_{L^\infty L^2}^2+(\nu |k|)^{\frac{1}{2}}\|e^{c\nu^{\frac{1}{2}}t} \omega_H\|_{L^2 L^2}^2\\
&+\nu \|e^{c\nu^{\frac{1}{2}}t} (\partial_y,|k|)\omega_H\|_{L^2 L^2}^2+|k| \| e^{c\nu^{\frac{1}{2}}t} u_H \|_{L^2 L^2}^2
\lesssim \|\Delta_k \omega_0\|_{L^2}^2.
\end{aligned}
\end{equation}
\end{lemma}

\begin{proof}
The proof is based on the resolvent estimates for $\mathcal{L}_k-\epsilon (\nu |k|)^{\frac{1}{2}}$, we only give the proof for the estimates without $e^{c\nu^{\frac{1}{2}}t}$, and the case with the weight can be obtained via perturbation argument with standard energy argument.

By the Gearhart-Pr\" uss type lemma in \cite{Wei} and the resolvent estimates in Proposition \ref{resolvent-slip}, we have 
$$\|\omega_H(t)\|_{L^2}\leq Ce^{-c_1(\nu |k|)^{\frac{1}{2}}t}\|\omega_0\|_{L^2},$$
which gives that 
$$\|\omega_H\|_{L^\infty L^2}^2+(\nu |k|)^{\frac{1}{2}}\|\omega_H\|_{L^2 L^2}^2\lesssim \|\omega_0\|_{L^2}^2.$$
The energy argument gives that 
\begin{equation}\nonumber
\begin{aligned}
\frac{1}{2}\frac{\mathrm{d}}{\mathrm{d}t}\|\omega_H\|_{L^2}^2+\nu \|(\partial_y,|k|)\omega_H\|_{L^2}^2 =0,
\end{aligned}
\end{equation}
which yields that 
\begin{equation}\nonumber
\begin{aligned}
\nu \|(\partial_y,|k|)\omega_H\|_{L^2 L^2}^2 \leq \|\omega_0\|_{L^2}^2.
\end{aligned}
\end{equation}

Next we give the estimate for $u_H$. 

If $\nu |k|^2\leq 2$, let $\omega_H^{(0)}$ be the solution to
\begin{equation}\nonumber
\begin{aligned}
&\partial_t \omega_H^{(0)}+ik(1-y^2)\omega_H^{(0)}+2ik\varphi_H^{(0)}=0,\\
&(\partial_y^2-|k|^2)\varphi_H^{(0)}=\omega_H^{(0)}, \varphi_H^{(0)}|_{y=\pm 1}=0, \ \omega_H^{(0)}|_{t=0}=\omega_0,
\end{aligned}
\end{equation}
due to $\omega_0\in H^1_0$, it holds that $\omega_H^{(0)}|_{y=\pm 1}=0$ for any $t\geq 0$.

Following the arguments of \cite{CDLZ}, let us introduce 
\begin{equation}\nonumber
\begin{aligned}
&\omega_H^{(1)}=g(t,y)\omega^{(0)}_H, \ g(t,y)=e^{-\nu |k|^2t-4\nu |k|^2y^2t^3/3-(\nu |k|)^{1/2}t},\\
& (\partial_y^2-|k|^2)\varphi_H^{(i)}=\omega_H^{(i)}, \quad \varphi_H^{(i)}(\pm 1)=0, \quad i=0,1.
\end{aligned}
\end{equation}
It is easy to find that
\begin{equation*}
\begin{aligned}
&\partial_t\omega_{H}^{(1)}+\nu |k|^2\omega_{H}^{(1)}+ (4\nu |k|^2y^2t^2+(\nu |k|)^{1/2})\omega_{H}^{(1)}\\
&\qquad \qquad +ik(1-y^2)\omega_{H}^{(1)}=-2ikg(t,y)\varphi_{H}^{(0)},\\
& \omega_H^{(0)}|_{t=0}=\omega_H^{(1)}|_{t=0}=\omega_0, \ \omega_H^{(1)}|_{y=\pm 1}=0.
\end{aligned}
\end{equation*}
Thus, we have
\begin{align*}
   &\partial_t\omega_{H}^{(1)}-\nu(\partial_y^2-|k|^2)\omega_{H}^{(1)} +ik(1-y^2)\omega_{H}^{(1)}+2ik\varphi_{H}^{(1)}\\
    &=-\nu\partial_y^2\omega_{H}^{(1)}- 4\nu |k|^2y^2t^2\omega_{H}^{(1)} -(\nu| k|)^{1/2}\omega_{H}^{(1)}-2ik(g\varphi_{H}^{(0)}-\varphi_{H}^{(1)}).
 \end{align*}

 Now we decompose $\omega_H$ as follows
\begin{align}\nonumber
   & \omega_H=\omega_H^{(1)}+\omega_H^{(2)},
\end{align}
where $\omega_H^{(2)}$ satisfies
\begin{equation*}
\left \{
\begin{array}{lll}
(\partial_t+\mathcal{L}_k)\omega_H^{(2)}=\nu\partial_y^2\omega_{H}^{(1)}+ 4\nu |k|^2y^2t^2\omega_{H}^{(1)} +|\nu k|^{1/2}\omega_{H}^{(1)}\\
  \qquad  \qquad\qquad\quad +2ik(g(t,y)\varphi_{H}^{(0)}-\varphi_{H}^{(1)}),\\
 (\partial_y^2-|k|^2) \varphi_H^{(2)}=\omega_H^{(2)},\
\omega_H^{(2)}|_{t=0}=0, \   \omega_H^{(2)}|_{y=\pm 1}=0,
\end{array}\right.
\end{equation*}

Following the arguments as in Lemma 4.2 of \cite{CDLZ}, we have 
\begin{align}\label{u-H-1}
|k|\|(\partial_y,|k|)\varphi_H^{(1)}\|_{L^2 L^2}^2\lesssim \|(\partial_y,|k|)\omega_0\|_{L^2}^2.
 \end{align}

Next, we estimate $\|(\partial_y,|k|)\varphi_H^{(2)}\|_{L^2 L^2}$. Let us recall from Lemma 4.2 in \cite{CDLZ} (here we set $k_1=k$) that 
\begin{align*}
\partial_y^2\omega^{(1)}_H 
&=\partial_y\omega_1^{(1)}+2ik t\omega_{H}^{(1)}+2ik ty(\omega_1^{(1)}+2ik ty\omega^{(1)}_{H})
 ,\\ \partial_y^2\omega_{H}^{(1)}&+4|k|^2y^2t^2\omega^{(1)}_{H} =\partial_y\omega_1^{(1)}+2ikt\omega^{(1)}_{H}+2ikty\omega_1^{(1)},\\
 &\omega_1=X{\omega}^{(0)}_H,\ \omega_1^{(1)}=X\omega^{(1)}_{H},\ X=\partial_y-2iky t.\\
 &\phi_1^{(1)}=g(t,y)\varphi_H^{(0)}-\varphi_H^{(1)},
\end{align*}
and we have 
\begin{align*}
   &\partial_t\omega^{(2)}_H-\nu(\partial_y^2-|k|^2)\omega_{H}^{(2)}
   +ik \big( (1-y^2)\omega_H^{(2)}+2\varphi^{(2)}_H\big)\\
   &=\nu(\partial_y\omega_1^{(1)}+2ikt\omega^{(1)}_{H}+2ik ty\omega_1^{(1)})+ |\nu k|^{1/2}\omega_{H}^{(1)}
   +2ik \phi_1^{(1)}.
\end{align*}
Applying Lemma \ref{sp-I} with $f_2=-\nu \omega_{1}^{(1)}$, $f_4=-2\nu ikty\omega_{1}^{(1)}-2\nu ikt\omega_H^{(1)}-(\nu |k|)^{\frac{1}{2}}\omega_H^{(1)}$ and $f_3=-2ik\phi_1^{(1)}$ that
\begin{align*}
   \|(\partial_y,|k|)\varphi_H^{(2)}\|_{L^2 L^2}^2 \lesssim& \nu^{-1} |k|^{-1} \|\nu \omega_{1}^{(1)}\|_{L^2 L^2}^2 +|k|^{-2} \|k (\partial_y,|k|)\phi_1^{(1)}\|_{L^2 L^2}^2\\
   &+\nu^{-\frac{1}{2}} |k|^{-\frac{3}{2}} \bigg(\|\nu kty\omega_{1}^{(1)}\|_{L^2 L^2}^2+\|(|\nu kt|+(\nu |k|)^{\frac{1}{2}})\omega_H^{(1)}\|_{L^2 L^2}^2\bigg)\\
   \lesssim& |k|^{-1} \|\Delta_k \omega_0\|_{L^2}^2.
\end{align*}

Therefore, we conclude that 
$$ |k|\|u_H\|_{L^2 L^2}^2 \lesssim \|\Delta_k \omega_0\|_{L^2}^2,$$
which completes the proof for case of $\nu |k|^2 \leq 2$.

If $\nu |k|^2\geq 2$, then it holds that 
\begin{align*}
|k|^2 \|u_H\|_{L^2 L^2}^2\lesssim \nu |k|^2 \|\omega_H\|_{L^2 L^2}^2 \lesssim \|\omega_0\|_{L^2}^2,
\end{align*}
which completes that proof.
\end{proof}

\section{Nonlinear stability}\label{pf-main}
For the nonlinear problem with Navier-slip boundary condition, let us recall that \begin{equation}\label{nonlinear-eq}
\left \{
\begin{array}{lll}
\partial_t \omega_k +\mathcal{L}_k\omega_k=-ik(f_1)_k-\partial_y (f_2)_k,\\
(\partial_y^2-k^2)\varphi_k=\omega_k, \\
 \varphi_k(\pm 1)=\omega_k(\pm 1)=0,\\
\omega_k|_{t=0} =\omega_{0,k}(k,y),\end{array}
\right.
\end{equation}
where 
$$\mathcal{L}_k:=-\nu (\partial_y^2-k^2)+ik(1-y^2)+2ik(\partial_y^2-|k|^2)^{-1}$$
and 
$$(f_1)_k=-\sum_{l\in\mathbb{Z}} (u_1)_l \omega_{k-l}, \ (f_2)_k=-\sum_{l\in\mathbb{Z}} (u_2)_l \omega_{k-l}.$$

For simplicity, we define $\overline{\omega_0}:=\omega_{0,0}(0,y)$ for the zero mode part of the initial data.

Recall that for $k\not=0,$  
\begin{align*}
\mathcal{E}_k=& (\nu|k|)^{\frac{1}{4}}\|e^{c\nu^{\frac{1}{2}} t} \omega_k\|_{L^2 L^2}+\nu^{\frac{1}{2}}|k|\|e^{c\nu^{\frac{1}{2}} t} \omega_k\|_{L^2 L^2}\\
&+\|e^{c\nu^{\frac{1}{2}} t}\omega_k\|_{L^\infty L^2}+|k|^{\frac{1}{2}} \|e^{c\nu^{\frac{1}{2}} t} u_k\|_{L^2 L^2},
\end{align*}
and $\mathcal{E}_0=\|\overline{\omega}\|_{L^\infty L^2}$ for $k=0$.

Now we are on a position to give the proof of Theorem \ref{nonlinear-slip}.

\begin{proof}[Proof of Theorem \ref{nonlinear-slip}]
The global well-posedness of 2-D Navier-Stokes equation around the Poiseuille flow can be obtained by classical theory for Navier-Stokes equation. So we only focus on the stability estimates in Theorem \ref{nonlinear-slip}.

First let us treat with the zero modes. If $k=0$, by the divergence-free condition, we have $\partial_y (u_2)_0=0$, which implies by $(u_2)_0|_{y=\pm 1}=0$ that $ (u_2)_0=0$. Therefore, the equation for zero modes is reduced to 
$$\partial_t (u_1)_0-\nu \partial_y^2  (u_1)_0=-\partial_y (f_2)_0.$$
By the similar arguments as for 2-D Couette flow \cite{CLWZ}, we have  
\begin{align}\label{0mode}
\mathcal{E}_0^2\lesssim \nu^{-1}\|(f_2)_0\|_{L^2L^2}^2+\|\overline{\omega_0}\|_{L^2}^2.
\end{align}

For $\mathcal{E}_k$, by the space-time estimates in Proposition \ref{sp-slip-0}, one has 
\begin{equation}\nonumber
\begin{aligned}
\mathcal{E}_k 
 \lesssim  \|\Delta_k \omega_{0,k}\|_{L^2}+ \nu^{-\frac{3}{8}} |k|^{\frac{3}{8}}  \|e^{c\nu^{\frac{1}{2}} t} (f_1)_k\|_{L^2 L^2}+\nu^{-\frac{1}{2}}\| e^{c\nu^{\frac{1}{2}} t}(f_2)_k\|_{L^2 L^2}.
\end{aligned}
\end{equation}
where we have used the fact that $|k| \min\{(\nu |k|)^{-\frac{1}{4}}, (\nu |k|^2)^{-\frac{1}{2}}\}\leq  \nu^{-\frac{3}{8}} |k|^{\frac{3}{8}}  $.

For $k\not=0$, it holds that 
\begin{equation}\label{f-2-k}
\begin{aligned}
\|e^{c\nu^{\frac{1}{2}} t} (f_2)_k\|_{L^2 L^2}\leq& \sum_{l\in \mathbb{Z}}\| e^{c\nu^{\frac{1}{2}} t} (u_2)_l \omega_{k-l}\|_{L^2 L^2}\\
\leq& \sum_{l\in \mathbb{Z}}\| e^{c\nu^{\frac{1}{2}} t}(u_2)_l\|_{L^2 L^\infty}\| e^{c\nu^{\frac{1}{2}} t} \omega_{k-l}\|_{L^\infty L^2}\\
\leq &\sum_{l\in \mathbb{Z}}\|e^{c\nu^{\frac{1}{2}} t} (u_2)_l\|_{L^2 L^2}^{\frac{1}{2}}\|e^{c\nu^{\frac{1}{2}} t} \partial_y (u_2)_l\|_{L^2 L^2}^{\frac{1}{2}} \mathcal{E}_{k-l}\\
\leq &\sum_{l\in \mathbb{Z}} (|l|^{-\frac{1}{2}}\mathcal{E}_l)^{\frac{1}{2}}( |l|\|e^{c\nu^{\frac{1}{2}} t} (u_1)_l\|_{L^2 L^2})^{\frac{1}{2}} \mathcal{E}_{k-l}\\
\leq &\sum_{l\in \mathbb{Z}} (|l|^{-\frac{1}{2}}\mathcal{E}_l)^{\frac{1}{2}}(|l|^{-\frac{1}{2}} |l|\mathcal{E}_l)^{\frac{1}{2}} \mathcal{E}_{k-l}\lesssim \sum_{l\in \mathbb{Z}} \mathcal{E}_l \mathcal{E}_{k-l}
\end{aligned}
\end{equation}
and 
\begin{equation}\label{f-1-k-0}
\begin{aligned}
\|e^{c\nu^{\frac{1}{2}} t} (f_1)_k\|_{L^2 L^2}\leq& \|\overline{u}_1\|_{L^\infty L^\infty}\|e^{c\nu^{\frac{1}{2}} t} \omega_k\|_{L^2 L^2}+\|e^{c\nu^{\frac{1}{2}} t} (u_1)_k\|_{L^2 L^\infty}\|\overline{\omega}\|_{L^\infty L^2}\\
&+\sum_{l\in \mathbb{Z}\setminus \{0,k\}} \| e^{c\nu^{\frac{1}{2}} t} (u_1)_l\|_{L^\infty L^\infty}\|e^{c\nu^{\frac{1}{2}} t}\omega_{k-l}\|_{L^2 L^2}.
\end{aligned}
\end{equation}
It remains to estimate the terms on the right-hand side of \eqref{f-1-k-0}.

For $l\not=0,k$, it holds that $|k|\lesssim |l||k-l|$, then 
\begin{equation}\nonumber
\begin{aligned}
\sum_{l\in \mathbb{Z}\setminus \{0,k\}}& \| e^{c\nu^{\frac{1}{2}} t} (u_1)_l\|_{L^\infty L^\infty}\|e^{c\nu^{\frac{1}{2}} t} \omega_{k-l}\|_{L^2 L^2}\\
\lesssim & \sum_{l\in \mathbb{Z}\setminus \{0,k\}}  \|e^{c\nu^{\frac{1}{2}} t}  (u_1)_l\|_{L^\infty L^2}^{\frac{1}{2}} \|e^{c\nu^{\frac{1}{2}} t} \partial_y(u_1)_l\|_{L^\infty L^2}^{\frac{1}{2}}\\
&\times \|e^{c\nu^{\frac{1}{2}} t} \omega_{k-l}\|_{L^2 L^2}^{\frac{5}{6}}\|e^{c\nu^{\frac{1}{2}} t} \omega_{k-l}\|_{L^2 L^2}^{\frac{1}{6}}\\
\lesssim& \sum_{l\in \mathbb{Z}\setminus \{0,k\}} |l|^{-\frac{1}{2}}\mathcal{E}_l (\nu |k-l|)^{-\frac{1}{4}\cdot \frac{5}{6}} (\nu^{\frac{1}{2}} |k-l|)^{-\frac{1}{6}}\mathcal{E}_{k-l}\\ 
\lesssim& \nu^{-\frac{7}{24}} |k|^{-\frac{3}{8}} \sum_{l\in \mathbb{Z}\setminus \{0,k\}} |l|^{-\frac{1}{8}}\mathcal{E}_l \mathcal{E}_{k-l}
\end{aligned}
\end{equation}
and 
\begin{equation}\nonumber
\begin{aligned}
 \|&\overline{u}_1\|_{L^\infty L^\infty}\|e^{c\nu^{\frac{1}{2}} t} \omega_k\|_{L^2 L^2}+\|e^{c\nu^{\frac{1}{2}} t}(u_1)_k\|_{L^2 L^\infty}\|\overline{\omega}\|_{L^\infty L^2}\\
\lesssim&   \|\overline{\omega}\|_{L^\infty L^2 } (\nu |k|)^{-\frac{1}{4}\cdot \frac{5}{6}} (\nu^{\frac{1}{2}} |k|)^{-\frac{1}{6}}\mathcal{E}_{k}\\
&+\|e^{c\nu^{\frac{1}{2}} t}(u_1)_k\|_{L^2 L^2}^{\frac{1}{2}} \|e^{c\nu^{\frac{1}{2}} t}(\omega)_k\|_{L^2 L^2}^{\frac{1}{2}} \|\overline{\omega}\|_{L^\infty L^2}\\
\lesssim&\nu^{-\frac{7}{24}} |k|^{-\frac{3}{8}} \mathcal{E}_0 \mathcal{E}_{k}.
\end{aligned}
\end{equation}
Thus, we have 
\begin{equation}\label{f-1-k}
\begin{aligned}
\|e^{c\nu^{\frac{1}{2}} t} (f_1)_k\|_{L^2 L^2} \lesssim \nu^{-\frac{7}{24}} |k|^{-\frac{3}{8}}  \sum_{l\in \mathbb{Z}}  \mathcal{E}_l  \mathcal{E}_{k-l}
\end{aligned}
\end{equation}

Combining \eqref{0mode}, \eqref{f-2-k} and \eqref{f-1-k}, we have 
\begin{equation}\nonumber
\begin{aligned}
\mathcal{E}_k 
 \lesssim&  \|\Delta_k \omega_{0,k}\|_{L^2}+\nu^{-\frac{3}{8}} |k|^{\frac{3}{8}} \|e^{c\nu^{\frac{1}{2}} t} (f_1)_k\|_{L^2 L^2}+\nu^{-\frac{1}{2}}\|e^{c\nu^{\frac{1}{2}} t} (f_2)_k\|_{L^2 L^2}\\
 \lesssim&  \|\Delta_k \omega_{0,k}\|_{L^2}+\nu^{-\frac{2}{3}}  \sum_{l\in \mathbb{Z}}  \mathcal{E}_l  \mathcal{E}_{k-l},\\
 \mathcal{E}_0\lesssim& \|\overline{\omega}_0\|_{L^2}+\nu^{-\frac{1}{2}}  \sum_{l\in \mathbb{Z}\setminus \{0\}}  \mathcal{E}_l  \mathcal{E}_{-l},
\end{aligned}
\end{equation}
which gives that 
\begin{equation}\nonumber
\begin{aligned}
\sum_{k\in\mathbb{Z}} \mathcal{E}_k 
  \lesssim  \sum_{k\in\mathbb{Z}} \|\Delta_k \omega_{0,k}\|_{L^2}+\nu^{-\frac{2}{3}}  \sum_{k\in\mathbb{Z}} \sum_{l\in \mathbb{Z}}  \mathcal{E}_l  \mathcal{E}_{k-l}.
  \end{aligned}
\end{equation}

Due to $\|u_0\|_{H^{\frac{7}{2}+}} \leq \epsilon_1 \nu^{2/3}$ with 
$$    \sum_{k\in\mathbb{Z}} \|\Delta_k \omega_{0,k}\|_{L^2} \leq \epsilon_1 C\nu^{\frac{2}{3}}.$$
If $\epsilon_1$ is suitably small, one can deduce from continuous argument that 
\begin{equation}\nonumber
\begin{aligned}
\sum_{k\in\mathbb{Z}} \mathcal{E}_k 
\leq  \epsilon_1 C\nu^{\frac{2}{3}}.
  \end{aligned}
\end{equation}

The proof is completed.
\end{proof}

\smallskip
{\bf Acknowledgment.}

The authors thank to Prof. Zhifei Zhang for valuable discussions and suggestions. S. Ding is supported by the Key Project of National Natural Science Foundation of China under Grant 12131010, National Natural Science Foundation of China under Grant 12271032. Z. Lin is partially supported by National Natural Science Foundation of China under Grant 11971009.

\bigskip

\end{document}